\documentclass{amsart}

\usepackage{amssymb}
\usepackage{amsmath}

\newtheorem{theorem}{Theorem}[section]
\newtheorem{proposition}[theorem]{Proposition}
\newtheorem{corollary}[theorem]{Corollary}
\newtheorem{lemma}[theorem]{Lemma}

\theoremstyle{definition}
\newtheorem{definition}[theorem]{Definition}
\newtheorem{example}[theorem]{Example}

\theoremstyle{remark}
\newtheorem{remark}[theorem]{Remark}

\numberwithin{equation}{section}

\newcommand{\dist}{\ensuremath{\mathrm{dist} }}

\newcommand{\tub}{\ensuremath{\mathrm{Tub} }}
\newcommand{\F}{\ensuremath{\mathcal{F}}}
\newcommand{\singularF}{\ensuremath{\mathcal{X}_{F}}}

\newcommand{\dank}{\textsf{Acknowledgments.\ }}

\newcommand{\metric}{\ensuremath{ \mathtt{g} }}

\newcommand{\afir}{\textsf{Claim}\, }

\begin{document}


\title[On closed geodesics in $M/\F$]{On closed geodesics in  the leaf spaces of singular Riemannian foliations}

\author{Marcos M. Alexandrino}

\author{Miguel Angel Javaloyes }


\address{Marcos M. Alexandrino \hfill\break\indent 
Instituto de Matem\'{a}tica e Estat\'{\i}stica\\
Universidade de S\~{a}o Paulo, \hfill\break\indent
 Rua do Mat\~{a}o 1010,05508 090 S\~{a}o Paulo, Brazil}
\email{marcosmalex@yahoo.de, malex@ime.usp.br}

\address{Miguel Angel Javaloyes \\ 
Departamento de Geometr\'{\i}a y Topolog\'{\i}a.\hfill\break\indent
Facultad de Ciencias, Universidad de Granada.\hfill\break\indent
Campus Fuentenueva s/n, 18071 Granada, Spain}
\email{majava@ugr.es, ma.javaloyes@gmail.com}

\thanks{The first author was  supported by CNPq and partially supported by FAPESP. The second author was partially supported by Regional J.
Andaluc\'{\i}a Grant P06-FQM-01951 and by Spanish MEC Grant MTM2007-64504.}







\begin{abstract}

In this paper  we survey on some recent results on Riemannian orbifolds and singular Riemannian foliations and combine them to conclude the existence of closed  geodesics in the leaf space of some classes of singular Riemannian foliations (s.r.f.), namely s.r.f. that  admit sections or   have no horizontal conjugate points. We also investigate the shortening process with respect to Riemannian foliations. 



\end{abstract}


\maketitle

\section{Introduction}

In \cite[Theorem 5.1.1, Remark 5.1.2]{Haefliger} Guruprasad and Haefliger proved the existence of closed  geodesics in Riemannian compact orbifolds $Q$ (recall Definition \ref{definition-Riemannian-orbifold} and Definition \ref{definition-ClosedGeodesicOrbifold}) in the following cases:
\begin{enumerate}
\item[(1)] $Q$ is not developable (not good orbifold), 
\item[(2)] $Q$ is a good orbifold $\Sigma/W$ and $W$  has an element of infinite order or is finite.
\end{enumerate}

The aim of this paper is twofold. The first one is to survey on
 some recent results on Riemannian orbifolds and singular Riemannian foliations and combine them to conclude the existence of closed  geodesics in the leaf space of some  classes of singular Riemannian foliations.  
We start by recalling some concepts of orbifold theory and   by giving    an alternative  proof of item (2) of Guruprasad and Haefliger's theorem (see Theorem \ref{theorem-ClosedGeodesicGoodOrbifold}). Then we recall some facts about singular Riemannian foliations and,
  using the result of Guruprasad and Haefliger \cite{Haefliger},   Alekseevsky \emph{et al}.~   \cite{Alekseevsky} and Lytchak \cite{Lytchak},  conclude the existence of closed geodesics in $M/\F$ when $\F$ is a singular Riemannian foliation with closed embedded leaves on a simply connected complete Riemannian manifold $M$ and  $M/\F$ is a compact orbifold (see Theorem \ref{theorem-geodesicafechadaM/F}). In particular, we show the existence of  a closed geodesic of the orbifold $M/\F$ when $\F$ admits sections (e.g., the partition by orbits of a polar action) or $\F$ has no horizontal conjugate points (e.g., the  partition by orbits of  a variationally complete action),  $M/\F$ is compact and $M$ is simply connected (see Corollary \ref{corollary-s.r.f.s-ouno  horizontal conjugate points-geodesicafechada}).
 
The second aim of this paper is to  construct  the shortening process with respect to Riemannian foliations (see Section 3 and Theorem \ref{teo-RF-PeriodicGeodesic}). This 
 provides an algorithm to find closed geodesics in  some special Riemannian  orbifolds.
 With this  technique  we prove that every compact orbifold with nontrivial (topological) fundamental group admits  a closed geodesic (see Corollary \ref{corollary-orbifoldComapcto-grupofundamentalnaotrivial}). In particular  $M/\F$ admits  a closed geodesic if the fundamental group of the compact topological space $|M/\F|$ is nontrivial and $\F$ is a s.r.f.  that admits sections or $\F$ has no horizontal conjugate points (see Corollary \ref{corollary-GeodFechadaEspfolhasNaoSimplesmenteConexo}).

This paper is organized as follows.
In Section 2 we present and prove   Theorem \ref{theorem-ClosedGeodesicGoodOrbifold}, Theorem \ref{theorem-geodesicafechadaM/F} and Corollary \ref{corollary-s.r.f.s-ouno  horizontal conjugate points-geodesicafechada}.
In Section 3 we present the  shortening process with respect to Riemannian foliations 
and prove Theorem \ref{teo-RF-PeriodicGeodesic}  and Corollaries \ref{corollary-orbifoldComapcto-grupofundamentalnaotrivial} and \ref{corollary-GeodFechadaEspfolhasNaoSimplesmenteConexo}.
We also  include an appendix where we use the shortening process to give another proof of Theorem \ref{theorem-ClosedGeodesicGoodOrbifold}.

\dank 
The authors are grateful to Dr. Alexander Lytchak for very useful suggestions.




\section{Orbifolds and Riemannian foliations}

In this section, we  recall some definitions and results, state  
Theorem \ref{theorem-ClosedGeodesicGoodOrbifold}, Theorem \ref{theorem-geodesicafechadaM/F} and Corollary \ref{corollary-s.r.f.s-ouno  horizontal conjugate points-geodesicafechada} and give concise proofs for these results.

We start by recalling some facts about orbifolds. More details can be found in  Salem \cite[Appendix D]{Molino}, 
in Guruprasad and Haefliger \cite{Haefliger} or in  Moerdijk and  Mr\v{c}un \cite{Moerdijk}.



\begin{definition}[Riemannian pseudogroup]
Let $\Sigma$ be a Riemannian manifold, not necessarily connected. A \emph{pseudogroup} $W$ of isometries of $\Sigma$ is a collection of local isometries $w:U\rightarrow V,$ where $U$ and $V$ are open subsets of $\Sigma$ such that:
\begin{enumerate}
\item If $w\in W$ then $w^{-1}\in W.$
\item If $w:U\rightarrow V$ and $\tilde{w}:\tilde{U}\rightarrow\tilde{V}$ belong to $W,$ then 
$\tilde{w}\circ w:w^{-1}(\tilde{U})\rightarrow \tilde{w}(w^{-1}(\tilde{U}))\subset\tilde{V}$ also belongs to $W,$ if $V\cap \tilde{U}\neq\emptyset$.
\item If $w:U\rightarrow V$ belongs to $W,$ then its restriction to each open subset $\tilde{U}\subset U$ also belongs to $W.$
\item If $w:U\rightarrow V$ is an isometry between open subsets of $\Sigma$ which coincides in a neighborhood of each point of $U$ with an element of $W,$ then $w\in W.$
\end{enumerate} 
\end{definition}

\begin{definition}
Let $A$ be a family of local isometries of $\Sigma$ containing the identity map of $\Sigma.$ 
The pseudogroup obtained by taking the inverses of the elements of $A,$ the restrictions to open sets of elements of $A,$ as well as their compositions and their unions, is called the {\it pseudogroup generated by $A.$}
\end{definition}

An important example of a Riemannian pseudogroup is the holonomy pseudogroup of a Riemannian foliation, whose definition we now recall.
\begin{definition}
\label{example-holohomypseudogroup}
 Let $\F$ be a    foliation of codimension $k$ on a Riemannian manifold $(M,\metric).$ Then $\F$ is a \emph{Riemannian foliation} if it can be  described by an open cover $\{U_{i}\}$ of $M$ with Riemannian submersions $f_{i}:(U_{i},\metric)\rightarrow(\sigma_{i},b)$ (where $\sigma_{i}$ is a submanifold of dimension $k$) such that there are isometries $w_{i,j}:f_{i}(U_{i} \cap U_{j})\rightarrow f_{j}(U_{j}\cap U_{i})$ with $f_{j}=w_{i,j}\circ f_{i}.$ The elements $w_{i,j}$ acting on $\Sigma=\amalg \sigma_{i}$ generate a pseudogroup of isometries of $\Sigma$  called the \emph{holonomy pseudogroup of $\F$}.
\end{definition}

\begin{definition}[Riemannian orbifold]
\label{definition-Riemannian-orbifold}
One can define a $k$-dimensional Riemannian orbifold as an equivalence class of pseudogroups $W$ of isometries on a Riemannian manifold $\Sigma$ (dimension of $\Sigma$ is equal to $k$) verifying the following conditions:
\begin{enumerate}
\item The space of orbits $|\Sigma/ W|$ is Hausdorff.
\item For each $x\in\Sigma,$ there exists an open neighborhood $U$ of $x$ in $\Sigma$ such that the restriction of $W$ to $U$ is generated by a finite group of isometries of $U.$
\end{enumerate}
In addition, if $W$ is a discrete subgroup of isometries of $\Sigma$ whose action on $\Sigma$ is proper,  then $\Sigma/W$ is said  \emph{good} (for definition and properties of proper actions see e.g. Duistermaat and Kolk \cite{Duistermaat}).
\end{definition}

\begin{remark}
Let $\Sigma/W$ be a Riemannian good orbifold. Since the action $W\times\Sigma\rightarrow\Sigma$ is proper, one can conclude that $W$ is a closed subgroup of isometries of $\Sigma$ with discrete topology and the isotropy group $W_{x}$ is finite for each $x\in\Sigma$. 
\end{remark}
\begin{remark}\label{closedembedded}
An important example of a Riemannian orbifold is the space of leaves $M/\F$, where $M$ is a Riemannian manifold and $\F$ is a Riemannian foliation on $M$ with closed embedded leaves. In fact $M/\F$ turns out to be isomorphic to $\Sigma/W$, where  $\Sigma$ and $W$ were  presented in Definition \ref{example-holohomypseudogroup}. This is proved in   Molino \cite[Proposition 3.7]{Molino} when the leaves are compact. In order to prove  the case where the leaves are closed  and embedded,   it suffices to  generalize Lemma 3.7  in \cite{Molino} (using e.g. Claim 1 of Proposition 2.18 \cite{Alex6}).
The proof, mutatis mutandis, now follows the proof of Molino.
\end{remark}

\begin{remark}
\label{remark-superResumoAntigoApendice}
There exists  a reciprocal result, namely  each   Riemannian  orbifold $\Sigma/W$ is the space of leaves of a Riemannian foliation with compact leaves. In fact    Moerdijk and Mr\v{c}un \cite[Proposition 2.23]{Moerdijk} proved that if $U(E)$ is the unitary frame bundle of the complexification of the tangent bundle of $\Sigma$, then $U(E)/W$ admits a foliation $\F^{u}$ such that $\Sigma/W$ is the orbifold  $(U(E)/W)/\F^{u}$. 
Using the Riemannian conection of $\Sigma$ one can induce a distribution in $U(E)$ and in $U(E)/W$ and then find the  appropriate transverse metric such that the plaques of $\F^{u}$ can be described by local submersions.
\end{remark}

Given a pseudogroup, and in particular an orbifold, we can define a fundamental group as we now recall.

\begin{definition}
\label{definition-wloop}
A \emph{$W$-loop} with base point $x_{0}\in\Sigma$ is defined by    
\begin{enumerate}
\item a sequence $0=t_{0}<\cdots<t_{n}=1,$    
\item continuous paths $c_{i}:[t_{i-1},t_{i}]\rightarrow \Sigma,$ $1\leq i\leq n,$  
\item elements $w_{i}\in W$ defined in a neighborhood of $c_{i}(t_{i})$ for $1\leq i \leq n$ such that $c_{1}(0)=w_{n}c_{n}(1)=x_{0}$ and $w_{i}c_{i}(t_{i})=c_{i+1}(t_{i}),$ where $1\leq i\leq n-1.$
\end{enumerate}
A \emph{subdivision} of  a $W$-loop is obtained by adding new points to the interval $[0,1],$ by taking the restriction of the $c_{i}$ to these new intervals and $w=id$ at the new points.
\end{definition}


\begin{definition}
\label{definition-equivalenceWloops}
Two $W$-loops are \emph{equivalent} if there exists a subdivision common to the loops represented by $(c_{i},w_{i})$ and $(\tilde{c}_{i},\tilde{w}_{i})$ and elements $g_{i}\in W$ 
defined in a neighborhood of the path $c_{i}$ such that
\begin{enumerate}
\item  $g_{i}\circ c_{i}=\tilde{c}_{i},$ $1\leq i\leq n$,
\item $\tilde{w}_{i}\circ g_{i}$ and $g_{i+1}\circ w_{i}$ have the same germ at $c_{i}(t_{i}),$ $1\leq i\leq n-1,$
\item $\tilde{w}_{n}\circ g_{n}$ has the same germ at $c_{n}(1)$ as $w_{n}.$
\end{enumerate}
\end{definition}

\begin{definition}
A \emph{deformation} of a $W$-loop represented by $(c_{i},w_{i})$ is given by continuous deformations $c_{i}(s,\cdot)$ of the paths $c_{i}=c_{i}^{0}:[t_{i-1},t_{i}]\rightarrow \Sigma,$ such that $(c_{i}(s,\cdot),w_{i})$ represents a $W$-loop.
\end{definition}

\begin{definition}
Two $W$-loops are in the same \emph{homotopy class} if one can be obtained from the other by a series of subdivisions, equivalences and deformations.
The homotopy classes of $W$-loops based at $x_{0}\in \Sigma$ form a group $\pi_{1}(W,x_{0})$ called \emph{fundamental group of the pseudogroup} $W$ at the point $x_{0}.$
\end{definition}

\begin{remark}
If the orbit space $\Sigma/W$ is connected, then there exists an isomorphism, defined up to conjugation, between $\pi_{1}(W,x)$ and $\pi_{1}(W,y)$ for $x,y$ in $\Sigma.$ Thus we will write just $\pi_1(W)$ when convenient.
\end{remark}

\begin{definition}
If  $\Sigma/W$ is a connected orbifold, its {\it fundamental group}  $\pi(\Sigma/W)$  is defined as the fundamental group $\pi_{1}(W)$ of the pseudogroup $W$.
\end{definition}

\begin{remark}
The fundamental group of an orbifold $\Sigma/W$ is not the same as the fundamental group of the topological space $|\Sigma/W|$. One of the differences lies in item (2) of Definition \ref{definition-equivalenceWloops}. To understand this claim, consider $W$ the group generated by the reflection in the line $\{x=0\}$ in $\mathbb{R}^{2}$. We note that the line that joins $(-1,0)$ to ($1,0)$ is a nontrivial element of the fundamental group of the orbifold $\mathbb{R}^{2}/W$. If we would drop the words \emph{have the same germ at} in item (2) of Definition \ref{definition-equivalenceWloops} we would conclude that this curve is equivalent to the concatenation of the line that joins $(-1,0)$ to $(0,0)$ with the line that joins $(0,0)$ to $(-1,0)$. This curve is cleary homotopic to a point and hence the curve that joins  $(-1,0)$ to ($1,0)$ would be equivalent to a point.  
\end{remark}


\begin{definition}[Closed geodesics in a Riemannian orbifold]
\label{definition-ClosedGeodesicOrbifold}
Let $\Sigma/W$ be a Riemannian orbifold. A closed geodesic in $\Sigma/W$ is defined as:
\begin{enumerate}
\item a sequence $0=t_{0}<\ldots<t_{n}=1,$    
\item nontrivial segments of geodesics  $\gamma_{i}:[t_{i-1},t_{i}]\rightarrow \Sigma,$ $1\leq i\leq n,$  
\item elements $w_{i}\in W$ defined in a neighborhood of $\gamma_{i}(t_{i})$ for $1\leq i \leq n$ such that $\gamma_{1}(0)=w_{n}\gamma_{n}(1)$, $\gamma_{1}'(0)=d w_{n}\gamma_{n}'(1)$, $w_{i}\gamma_{i}(t_{i})=\gamma_{i+1}(t_{i}),$ $d w_{i}\gamma_{i}'(t_{i})=\gamma_{i+1}'(t_{i}),$ 
where $1\leq i\leq n-1.$
\end{enumerate}
Usually  a closed geodesic in $\Sigma/W$ is denoted by $(\gamma_{i},w_{i}).$ 
\end{definition}

\begin{remark}
\label{remark-geodesica-orbifoldConexo}

In order to prove the existence of closed geodesics in each compact Riemannian good orbifold, it suffices to prove the existence of closed geodesics in each compact Riemannian good orbifold $\Sigma/W$, where $\Sigma$ is a complete connected Riemannian manifold and $W$ has infinite cardinality. 
In fact, let $\{\Sigma_{i}\}$ be the connected components of $\Sigma$ and $\{W_{i}\}$ the subgroups of $W$ that send each $\Sigma_{i}$ onto $\Sigma_{i}.$  
First assume that each $W_{i}$ has finite cardinality. Then the fact that $\Sigma/W$ is compact implies that each $\Sigma_{i}$ is compact. Indeed, consider a cover of $\Sigma=\cup \Sigma_{i}$ by normal balls $B_{r}(p_{\alpha})$. Since $\pi:\Sigma\rightarrow \Sigma/W$ is an open map, we have that $\{\pi(B_{r}(p_{\alpha}))\}$ is an open covering of the compact set $\Sigma/W.$ Therefore we can find a finite cover $\{\pi(B_{r}(p_{j}))\}$. The assumption that  the cardinality of each $W_{i}$ is finite implies that each $\Sigma_{i}$ is covered by a finite number of closed bounded sets $w \overline{B_{r}(p_{j})}$ and hence each $\Sigma_{i}$ is compact. In this case the existence of a closed geodesic in $\Sigma/W$ follows from the existence of a closed geodesic in the compact manifold $\Sigma_{i}.$ 
Now we assume that there exists a connected component $\Sigma_{i_{0}}$ and a subgroup $W_{i_{0}}$ with infinite cardinality.  As $\Sigma/W$ is compact, it follows that  $\Sigma_{i_{0}}/W_{i_{0}}$ is also compact and $\Sigma_{i_0}$ is complete. 
Finally note that the existence of  a nontrivial closed geodesic in $\Sigma_{i_{0}}/W_{i_{0}}$ implies directly the existence of a closed geodesic in $\Sigma/W.$  
\end{remark}

\begin{theorem}
\label{theorem-ClosedGeodesicGoodOrbifold}
Let $\Sigma$ be a connected complete Riemannian manifold and $W$ be an infinite discrete subgroup of isometries of $\Sigma$ whose action on $\Sigma$ is proper and such  that the good orbifold $\Sigma/W$ is compact.  Assume that there exists an element $w^{0}\in W$ that does not fix points (e.g. $w^{0}$ has infinite order).
Then there exists a nontrivial closed geodesic in the Riemannian good orbifold $\Sigma/W$. 

\end{theorem}
\begin{proof}
Let us first proof that   $\inf_{y\in \Sigma}{\rm d}(y,w^{0} y)>0$. Assume on the contrary that there exists a sequence $\{x_n\}$ in $\Sigma$ such that $\lim_{n\to\infty}{\rm d}(x_n,w^0x_n)$ goes to zero. As $\Sigma/W$ is compact, there exists a sequence $\{g_k\}$ in $W$ and $y\in\Sigma$ such that 
\[
\begin{array}{ccc}
\displaystyle
\lim_{k\to\infty} g_k x_k=y&\text{and}&
\displaystyle\lim_{k\to\infty} g_k w^0x_k= y
\end{array}
\]
up to subsequences. As the second limit coincides with $\lim_{k\to\infty} g_k w^0g^{-1}_kg_kx_k=y$ and the action of $W$ in $\Sigma$ is proper, we obtain that $g_k w^0 g^{-1}_k\to g\in W_y$, and  being the action $W$ discrete, it follows that there exists $k_0\in\mathbb{N}$ such that $g=g_{k_0}w^0g^{-1}_{k_0}$.  Thus, $g_{k_0}w^0g^{-1}_{k_0}y=y$, but in this case $w^0$ fixes $g^{-1}_{k_0}y$ contradicting our hypothesis. Once we have that $\inf_{y\in\Sigma} {\rm d}(y,w^0 y)>0$, a similar argument proves that the infimum is attained in some point $x\in\Sigma$. Let $\gamma:[0,1]\rightarrow \Sigma$ be the geodesic minimizing the distance from $x$ to $w^0x$, which exists because $\Sigma$ is complete. If we prove that $\tilde{\gamma}:[0,2]\rightarrow \Sigma$ given by
\[\tilde{\gamma}(t)=\begin{cases}
\gamma(t) &\text{if $t\in[0,1]$}\\
w^0\gamma(t-1)&\text{if $t\in(1,2]$}
\end{cases}\]
is a smooth geodesic, then its projection in $\Sigma/W$ will be a closed geodesic. Let $x'=\gamma(t')$ with $t'\in(0,1)$. Then
${\rm d}(x',w^0x')\leq {\rm d}(x',w^0x)+{\rm d}(w^0x,w^0x')={\rm d}(x',w^0x)+{\rm d}(x,x')={\rm d}(x,w^0x).$
As ${\rm d}(x,w^0x)$ attains the minimum of the translation length of $w^0$, the last inequalities must be in fact equalities and $\tilde{\gamma}$ must be smooth in $t=1$. As it is smooth  in the rest of points, we finally conclude that its projection is a closed geodesic of the orbifold $\Sigma/W$.

\end{proof}


In what follows we give a simple but important example that  illustrates the above theorem. It also allows us to see the difference between the classical problem of existence of closed geodesics in compact manifolds and the problem of existence of closed geodesics in compact good orbifolds.

\begin{example}
\label{example-theorem-ClosedGeodesicGoodOrbifold}
Let $\Sigma$ be the Euclidean space $\mathbb{R}^{n}$ and $W$ be an infinite Coxeter group of isometries of $\mathbb{R}^{n}$, i.e., the subgroup of isometries $W$ is generated by reflections in hyperplanes of  a family $\mathcal{H}$, the topology induced in $W$ from the group of isometries of $\mathbb{R}^{n}$ is discrete and the action on $\mathbb{R}^{n}$ is proper. Assume that $\mathcal{H}$  is invariant by the action of $W$.
Also assume that $W$ is irreducible and $\mathbb{R}^{n}/W$ is compact.  
It is known  that $W$ is an affine Weyl group, i.e., a semidirect product of a Weyl group and a group of translations (see Bourbaki \cite{Bou} Ch.~VI \S2 Proposition~8 and 
Remarque~1 on~p.180]), in particular $W$ has an element that does not fix points and hence satisfies item (b) of the theorem. It is also known that $\mathbb{R}^{n}/W$ is a simplex and hence a contractible space (see Bourbaki \cite{Bou}
Ch.~V \S3 Propositions~6,7, 8, and~10, and Remarque~1 on~p.86).  Therefore a compact good orbifold can be contractible. This does not happen with compact manifolds that always admit a nontrivial homotopy group. This topological property plays a fundamental role in the proof of Lyusternik and Fet about the existence of closed geodesics in compact manifolds (see Jost \cite{Jost}). It is also interesting to note that the fundamental group of the topological space $|\mathbb{R}^{n}/W|$ is trivial, but the fundamental group of the orbifold $\mathbb{R}^{n}/W$ is $W$, since $\mathbb{R}^{n}$ is simply connected (for  proofs see   Bridson and  Haefliger \cite[page 608]{BridsonHaefliger}). 

\end{example}

\begin{remark}
\label{remark-ExemplosPatologicos}
It is interesting to note that if there  existed an infinite torsion  finitely presented group,  then it would be possible to construct an example of a compact Riemannian good orbifold $\Sigma/W$ (that is not a manifold)  so that $W$  would not  necessarily satisfy the condition of Theorem \ref{theorem-ClosedGeodesicGoodOrbifold}. Nevertheless, as far  as the  authors know, the existence of such kind of group remains an open problem.  
\end{remark}

Closed geodesics in Riemannian orbifolds are related to horizontal periodic geo\-de\-sics of  Riemannian foliations  as we now explain.


\begin{definition}
\label{definition-folheacao}
 A partition $\F$ of a complete Riemannian manifold $M$ by connected immersed submanifolds (the \emph{leaves}) without self intersections   is said
\begin{enumerate}
\item[(1)]  a {\it  singular foliation},
if the module $\singularF$ of smooth vector fields on $M$ that are tangent at each point to the corresponding leaf acts transitively on each leaf. In other words, for each leaf $L$ and each $v\in TL$ with footpoint $p,$ there is $X\in \singularF$ with $X(p)=v$;
\item[(2)] a {\it  singular Riemannian foliation}, if it satisfies $(1)$ and it is \emph{transnormal}, i.e., every geodesic that is perpendicular at one point to a leaf remains perpendicular to every leaf that meets.
\end{enumerate}
\end{definition}
\begin{remark}
Let $\F$ be a singular Riemannian foliation. A leaf $L$ of $\F$ (and each point in $L$) is called \emph{regular} if the dimension of $L$ is maximal, otherwise $L$ is called {\it singular}. If all the leaves of $\F$ have the same codimension $k$, then $\F$ turns out to be a  Riemannian foliation of codimension $k$.
 \end{remark}
Typical examples of (singular) Riemannian foliations are the partition by orbits of an isometric action, by leaf closures of a Riemannian foliation (see \cite{Molino} and \cite{Alex4}), examples constructed by suspension of homomorphisms (see  \cite{Alex2,Alex4}),   and examples constructed by changes of metric and surgery (see \cite{AlexToeben}). 

An important property of Riemannian foliations is called  equifocality. In order to understand this concept, we need some preliminary definitions.

A \emph{Bott} or basic connection $\nabla$ of a foliation $\F$ is a connection of the normal bundle of the leaves  with $\nabla_XY=[X,Y]^{\nu\F}$ whenever $X\in \singularF$  and $Y$ is  a vector field of the normal bundle $\nu\F$ of the foliation. Here the superscript $\nu\F$ denotes projection onto $\nu\F$.

 A \emph{normal foliated vector field} is a normal field parallel with respect to the Bott connection. 
If we consider a local submersion $f$ which describes the plaques of $\F$ in a neighborhood of a point of $L$, then a normal foliated vector field is a normal projectable/basic vector field with respect to $f.$

The fundamental property of Riemannian foliations, called \emph{equifocality}  is described in the following proposition. 

\begin{proposition}
If $\xi$ is a normal parallel vector field (with respect to the Bott connection) along a curve $\beta:[0,1]\rightarrow L$  then the curve $t\mapsto \exp_{\beta(t)}(\xi)$ is  contained in the leaf $L_{\exp_{\beta(0)}(\xi)}$.
\end{proposition}
This property still holds even for singular Riemannian foliations  and implies that one  can reconstruct the  (singular) foliation by taking all parallel submanifolds of a (regular) leaf with trivial holonomy (see \cite{AlexToeben2}).

The equifocality allows us to introduce the concept of  parallel transport (with respect to the Bott connection) of  horizontal segments of geodesic.

\begin{definition}
\label{paralleltransport}
Let   $\beta:[a,b]\rightarrow L$ be a piecewise curve and $\gamma:[0,1]\rightarrow M$ a segment of horizontal geodesic  such that $\gamma(0)=\beta(a)$. Let $\xi_0$ be a vector of the normal space $\nu_{\beta(a)}L$  such that $\exp_{\gamma(0)}(\xi_0)=\gamma(1)$ and $\xi:[a,b]\rightarrow \nu L$ the parallel transport of $\xi_0$ with respect to the Bott connection along $\beta$.  
We define $\|_\beta(\gamma):=\hat{\gamma},$
where $\hat{\gamma}:[0,1]\rightarrow M$ is the segment of geodesic given by $s\rightarrow \hat{\gamma}(s)=\exp_{\beta(b)}(s\,\xi(b))$. We also set 
$\eta(\gamma,\beta):=\hat{\beta}$, where $\hat{\beta}$ is the curve contained in $L_{\gamma(1)}$ defined as $s\rightarrow\hat{\beta}(s)=\exp_{\beta(s)}(\xi(s))$.
\end{definition}

Due to the equifocality of $\F$, we can give an alternative definition of holonomy map of a Riemannian foliation.

\begin{definition}
\label{definition-holonomy-map}
Let $\beta:[0,1]\rightarrow L$ be a piecewise curve and $S_{\beta(i)}:=\{\exp_{\beta(i)}(\xi) | \xi\in \nu_{\beta(i)} L, \|\xi\|<\epsilon\}$ the \emph{slice} at $\beta(i)$, for $i=0,1$. Then a holonomy map $\varphi_{[\beta]}:S_{\beta(0)}\rightarrow S_{\beta(1)}$ is defined as 
$\varphi_{[\beta]}(x):=||_{\beta}\gamma(r)$, where $\gamma:[0,r]\rightarrow S_{\beta(0)}$ is the minimal segment of geodesic that joins $\beta(0)$ to $x$.
Since the Bott connection is locally flat, the parallel transport depends only on the homotopy class $[\beta]$. 
\end{definition} 

Using the   holonomy map of a Riemannian foliation  we can define horizontal periodic geodesics as follows.

\begin{definition}
\label{definition-periodicgeodesic}
Let $\F$ be a Riemannian foliation. A geodesic $\gamma$ is called  $\F$ \emph{horizontal periodic} if
\begin{enumerate}
\item[(a)] $\gamma$ is horizontal, i.e., is orthogonal to the leaves of $\F$,
\item[(b)] there exists $0<t_{1}$ such that $\gamma(t_{1})\in L_{\gamma(0)}$,
\item[(c)] there exists a holonomy map  $\varphi_{[\beta]}$ such that $ d \varphi_{[\beta]}\gamma'(0)= \gamma'(t_{1}).$
\end{enumerate}
If $t_{1}$ is the smallest positive number that satisfies (b) and (c) then $t_{1}$ is called the \emph{period} of $\gamma$.
\end{definition}

\begin{remark}
\label{proposition-horizontalPeriodicalGeodesic-periodic}
By the equifocality of Riemannian foliations  we can deduce that 
for each fixed $s$ and each $n\in \mathbb{Z}$ we have:
\begin{enumerate}
\item[(a)]  $\gamma(n t_{1}+s)\in L_{\gamma(s)}$;
\item[(b)] there exists a holonomy map  $\varphi_{[\beta_{n}]}$ such that $ d \varphi_{[\beta_{n}]}\gamma '(s)= \gamma '(n t_{1}+s).$
\end{enumerate}
\end{remark}

 As observed in Remark \ref{closedembedded}, the space of leaves of a Riemannian foliation with closed embedded leaves is isomorphic to a Riemannian orbifold.  Furthermore,  for each closed geodesic of the Riemannian orbifold $M/\F$ there exists a horizontal periodic geodesic and vice versa.  This implies in particular the next result. 

\begin{proposition}
A Riemannian foliation with closed embedded leaves $(M,\F)$ admits a horizontal periodic geodesic if and only if the orbifold $M/\F$ admits a closed geodesic.
\end{proposition}

In what follows we prove that   if $M$ is simply connected and $M/\F$ is  a compact  orfibold then $M/\F$ admits a closed geodesic, even if $\F$ is a singular Riemannian foliation (s.r.f. for short).

\begin{theorem}
\label{theorem-geodesicafechadaM/F}
Let $\F$ be a singular Riemannian foliation  with closed embedded leaves on a  Riemannian manifold $M$ with finite fundamental group. Assume that $M/\F$ is a compact orbifold. Then $M/\F$ admits a nontrivial closed geodesic.
\end{theorem}
\begin{proof}
 First we will prove the case where $M$ is simply connected. If $\F$ is a (regular) Riemannian foliation, according to Salem \cite[Appendix D]{Molino}, there exists a surjective homomorphism between $\pi_{1}(M)$ and the fundamental group of the holonomy of the foliation,  that coincides with the fundamental group of the orbifold $M/\F$. Therefore the fundamental group of the orbifold $M/\F$ is  trivial and hence it  cannot be a good orbifold (see Bridson and  Haefliger \cite[page 608]{BridsonHaefliger}). The result follows then from item (1) of the theorem of  Guruprasad and Haefliger
\cite{Haefliger}.
If $\F$ is a s.r.f. and $M/\F$ is  not developable, then the result also follows from item (1) of the theorem of Guruprasad and Haefliger \cite{Haefliger}.

Now if $\F$ is a s.r.f. and $M/\F$ is  a good orbifold, then according to  Lytchak \cite{Lytchak} $M/\F$ is a Coxeter orbifold.  
It follows from Alekseevsky et al. \cite[Theorem 6.4]{Alekseevsky} that any Coxeter orbifold is the Weyl chamber of a Riemannian Coxeter manifold     
$(\Sigma,W)$ i.e., $W$ is  a discrete subgroup of isometries of a complete Riemannian
manifold $\Sigma$, which is generated by disecting reflections. The fact that $\Sigma/W$ is compact implies that $W$ is finitely generated Coxeter group (see  \cite[Theorem 2.11 and Theorem 3.5]{Alekseevsky}).  It is known that any finitely generated Coxeter group $W$ has a torsion-free subgroup
of finite index (see Gonciulea \cite[Proposition 1.4]{Gonciulea}). Therefore $W$ is finite or $W$ has an element of infinite order. In both cases we have seen in Theorem \ref{theorem-ClosedGeodesicGoodOrbifold} and Remark \ref{remark-geodesica-orbifoldConexo} that $M/\F=\Sigma/W$ admits a closed geodesic.

 Now consider the case where $M$ has finite fundamental group. Let $\tilde{M}$ be the universal covering of $M$ . Then the foliation $(M,\F)$ induces naturally a foliation $(\tilde{M},\tilde{\F})$. As we have assumed that the fundamental group of $M$ is finite, it follows that $\tilde{M}/\tilde{\F}$ is compact and the leaves of $\tilde{\F}$ are closed and embedded. Thus we can apply the first part of the proof to obtain a closed geodesic in $\tilde{M}/\tilde{\F}$ that projects to a closed geodesic in $M/\F$ as desired.
\end{proof}
\begin{remark}
Note that in the  proof of the above theorem, we  show  that each compact Coxeter orbifold admits a closed geodesic.
\end{remark}
When $M$ is simply connected and the leaves of $\F$ are closed embedded
there are (at least) two special classes of singular Riemannian foliations such that $M/\F$ is  an  orbifold.

The first one is the class of \emph{singular Riemannian  foliations without horizontal conjugate points}. This concept was introduced by Lytchak and Thorbergsson \cite{LytchakThorbergsson1} and generalizes the definition of  variationally complete actions. $\F$ is \emph{without horizontal conjugate points} if the following is true for all leaves $L$ and all geodesics $\gamma$ meeting $L$ perpendicularly. Any $L$-Jacobi field $J$ along $\gamma$ that is tangent to a leaf of $\F$ different from $L$ is tangent to all leaves $\gamma$ passes through. It follows from Lytchak \cite[Theorem 1.2, Theorem 1.6]{Lytchak} that $M/\F$ is a Riemannian Coxeter  orbifold, if $M$ is simply connected and the leaves of $\F$ are closed embedded (see also Lytchak and Thorbergsson \cite[Theorem 1.7]{LytchakThorbergsson2}).

The other class is \emph{singular Riemannian foliations with sections} (s.r.f.s. for short). This concept was introduced by the first author \cite{Alex2}. Typical examples of singular Riemannian foliations with sections  are the partition by orbits of a polar  action, isoparametric foliations on space forms (some of them with inhomogeneous leaves) and  partitions by parallel submanifolds of an equifocal submanifold (see Terng and Thorbergsson \cite{TTh1} and Thorbergsson \cite{Th}).

A singular Riemannian foliation \emph{admits} sections if for each regular point $p$, the set $\Sigma:=\exp(\nu_{p}L)$ (\emph{section}) is a complete immersed submanifold that meets each leaf orthogonally.   

It was proved by Alexandrino and T\"{o}ben \cite[Theorem 1.6]{AlexToeben} that $M/\F$ is a Coxeter orbifold if $M$ is simply connected and the leaves of $\F$ are closed embedded (see also \cite{Alex5}). 

The above discussion and Theorem \ref{theorem-geodesicafechadaM/F} imply  the next corollary.

\begin{corollary}
\label{corollary-s.r.f.s-ouno  horizontal conjugate points-geodesicafechada}
Let $\F$ be a singular Riemannian foliation  with closed embedded leaves on a Riemannian manifold $M$ with finite fundamental group and such that $M/\F$ is compact. Assume that  $\F$ admits  sections or $\F$ has no  horizontal conjugate points. Then $M/\F$ admits a nontrivial closed geodesic.
\end{corollary}

\begin{remark}
The above corollary and  Myers' theorem imply that if $\F$ is a s.r.f.  that admits  sections or $\F$ has no  horizontal conjugate points on a complete Riemannian manifold with $\mathrm{Ric}\geq k>0$ (e.g. symmetric spaces of compact type) then $M/\F$ admits a closed geodesic. Therefore we have the existence of closed geodesics in the orbit spaces  of polar and variationally complete actions in  symmetric spaces of compact type, the usual space where these actions are studied.  
\end{remark}
\begin{remark} 
If $\F$ admits sections, then, due to the equifocality of $\F$,  the holonomy map $\varphi_{[\beta]}$ can be extended to include singular points (see \cite{Alex2}) and hence   Definition \ref{definition-periodicgeodesic} still makes sence for this class of singular  foliations.
 For a fixed section $\Sigma$ we can consider the pseudogroup $W_{\Sigma}$ generated by all  holonomy map $\varphi_{[\beta]}$ such that $\beta(0)$ and $\beta(1)$ belong to $\Sigma.$ This pseudogroup is called the \emph{Weyl pseudogroup}. It is possible to prove that $M/\F=\Sigma/W_{\Sigma}$ (see \cite{AlexToeben}).
It is easy to see then that the existence of closed  geodesics in $M/\F$ is equivalent to the existence of horizontal  geodesics in $M$ (as in Definition \ref{definition-periodicgeodesic}).
\end{remark}

We conclude this section suggesting some natural questions.
The first one is how closed geodesics of $M/\F$ can be used to study the
transverse geometry of a s.r.f. that admits sections or has no  horizontal conjugate points.
 It is also natural to ask for conditions under which $M/\F$ admits 
closed geodesics even  if $M/\F$ is not an orbifold. 
A naive approach to this last question would be to use 
a recent result (see \cite{Alex6}) that assures us
that $M/\F$ is always a Gromov-Hausdorff limit of a sequence of orbifolds
if the leaves of $\F$ are closed embedded and $M/\F$ is compact.


\section{Riemannian foliations and shortening process}

In this section we   study the shortening process with respect to Riemannian foliations and   prove Theorem \ref{teo-RF-PeriodicGeodesic} below. 
Theorem \ref{teo-RF-PeriodicGeodesic} assures the existence of horizontal periodic geodesics, assuming only topological conditions about the space $M$ and the foliation $\F$.



\begin{theorem}
\label{teo-RF-PeriodicGeodesic}
Let $\F$ be a Riemannian foliation with compact leaves on a compact Riemannian manifold $M$. 
Assume that either one of the two conditions below is satisfied: 
\begin{enumerate}
\item[(a)] there exists a loop $\alpha$ in $M$ that is not free homotopic to any  loop contained in any leaf.
\item[(b)] $\pi_{1}(M)$ admits a sequence of distinct elements in different classes of free homotopy  and the fundamental group of each leaf has finite cardinality.
\end{enumerate}
Then there exists an $\F$-horizontal periodic geodesic. 
In particular, if there exists a loop $\alpha$ in $M$ that satisfies the condition of item (a), 
then a subsequence  of iterations of double shortening of $\alpha$ 
converges to a nontrivial horizontal periodic geodesic. 
\end{theorem}

\begin{remark}\label{remark-obsa}
 We observe that item (a) of the last theorem is satisfied for example when there is an element of the fundamental group of $M$ with infinite order and the fundamental group of each leaf is finite.
\end{remark}


\begin{corollary}
\label{corollary-orbifoldComapcto-grupofundamentalnaotrivial}
Let $\Sigma/W$ be a Riemannian compact orbifold. Assume that the fundamental group of the topological space $|\Sigma/W|$ is nontrivial. Then there exists a closed geodesic in the orbifold $\Sigma/W$. 
\end{corollary}
\begin{proof}
It follows from Remark \ref{remark-superResumoAntigoApendice} that  there exists a Riemannian foliation $\F$ with compact leaves on a compact manifold $M$ such that $M/\F=\Sigma/W$. 

Let $\alpha$ be a nontrivial element of the fundamental group of the topological space $|\Sigma/W|$.

\afir: \emph{There exists a horizontal curve $\tilde{\alpha}$ whose projection in $M/\F$ is homotopic to $\alpha$}

In order to prove the claim note that for each point $x\in M/\F$ there exists a neighborhood $V_{x}$ such that for each point of $y\in V_{x}$ there exists only one segment that joins $y$ to $x$.
In fact this neighborhood is the image of a tubular neighborhood of a leaf in $M$ (see the definition of $\rho_0$ of  Subsection \ref{shortening}). 

Consider the curve $\alpha:[0,1]\rightarrow M/\F$. Then for each $t$ we can find a neighborhood $V_t$ and an interval $a_t\leq t\leq b_t$ such that $\alpha[a_t,b_t]\subset V_t$ and for $s\in [ a_t b_t ]$ there exists only one segment that joins $\alpha(t)$ to $\alpha(s)$.
Let $[\alpha(t),\alpha(s)]$ denotes this segment and $\star$ the concatenation of curves (see Remark \ref{remark-convention-concatenation}).
Note that the curve $\alpha|_{[t,b_t]}$ is homotopic (by  a homotopy that fixes  endpoints) to $[\alpha(t),\alpha(b_t)]$ by the family of curves
$\{\alpha|_{[s,b_t]}\star[\alpha(t),\alpha(s)]\}.$
Similarly $\alpha|_{[a_t,t]}$ is homotopic to $[\alpha(a_t),\alpha(t)].$ Therefore $\alpha[a_t,b_t]$ is homotopic to $\alpha|_{[t,b_t]}\star[\alpha(a_t),\alpha(t)]$ (by  a homotopy that fixes   endpoints). This fact and the fact that $[0,1]$ is compact imply that $\alpha$ is homotopic to a curve that is union of segments. Finally note that each segment can be lifted to a horizontal segment of geodesic in $M$
and by translation with respect to Bott connection we can construct the desired horizontal curve $\tilde{\alpha}$. This concludes 
the proof of the claim.

Now consider a curve $\beta$ in the leaf $L_{\tilde{\alpha}(0)}$ that joins $\tilde{\alpha}(1)$ with $\tilde{\alpha}(0)$. Note that the concatenation $\beta\star\tilde{\alpha}$ is not free homotopic to any  loop contained in any leaf. Otherwise we could project the homotopy between
$\beta\star\tilde{\alpha}$ and  a loop contained in a leaf and get a homotopy between $\alpha$ and a point of $\Sigma/W$. This  contradicts the fact that $\alpha$ is a nontrivial element of the fundamental group of the topological space $|\Sigma/W|$.

Finally it follows from item (a) of Theorem \ref{teo-RF-PeriodicGeodesic} that there exists a nontrivial horizontal periodic geodesic of $\F$ and hence a nontrivial closed geodesic of the orbifold $\Sigma/W=M/\F$.

\end{proof}

Let $\F$ be a s.r.f. with closed embedded leaves on a complete manifold. Assume  that  $\F$ admits sections or $\F$ has no horizontal conjugate points. Then Lytchak and Thorbergsson \cite{LytchakThorbergsson2} proved that $M/\F$ is an orbifold (not necessarly a good one). Therefore the above corollary implies the next one.

\begin{corollary}
\label{corollary-GeodFechadaEspfolhasNaoSimplesmenteConexo}
Let $\F$ be a s.r.f with closed embedded leaves on a complete Riemannian manifold $M$  and such that $M/\F$ is compact. Assume that $\F$ admits sections or $\F$ has no horizontal conjugate points. Also assume that the fundamental group of the topological space $M/\F$ is nontrivial. Then there exists a closed geodesic in the orbifold $M/\F$. 
\end{corollary}

\subsection{Shortening process}

\label{shortening} 

In this subsection we construct the shortening process with respect to a Riemannian foliation. We skip the proofs of lemmas since they follow from  standard arguments from the theory of foliations. 


\begin{remark}[Conventions]
\label{remark-convention-concatenation}
We will use two different concatenations of curves. We will denote by $*$ the curve obtained as the union of two curves $\alpha_1:[a,b]\rightarrow M$ and $\alpha_2:[b,c]\rightarrow M$, that is, the curve $\alpha_1*\alpha_2:[a,c]\rightarrow M$ that coincides with $\alpha_1$ in $[a,b]$ and with $\alpha_2$ in $[b,c]$. On the other hand, we will denote by $\star$ the concatenation of two curves $\alpha_1,\alpha_2:[a,b]\rightarrow M$, that is, the curve in $[a,b]$ such that $\alpha_2\star\alpha_1(s)=\alpha_1(2s-a)$ in $[a,a+(b-a)/2]$ and $\alpha_2\star\alpha_1(s)=\alpha_2(2s-b)$ in $[a+(b-a)/2,b]$. Moreover, given a curve $\alpha:[a,b]\rightarrow M$ we will denote by $\alpha^{-1}:[a,b]\rightarrow M$ the curve defined as $\alpha^{-1}(s)=\alpha(b+a-s)$.
\end{remark}


We also  need the notation below, which turns out to be very convenient  to  describe the curve shortening procedure.

\begin{definition}
\label{fclosed}
Let $\F$ be a Riemannian foliation on $(M,\metric)$ and  $\alpha$ and $\beta$ be two piecewise smooth curves $\alpha:[a,b]\rightarrow M$ and $\beta:[a,b]\rightarrow M$ such that the endpoints of $\alpha$ belong to the same leaf $L_{\alpha(a)}$ and the image of $\beta$ is contained in $L_{\alpha(a)}$ with $\beta(a)=\alpha(a)$ and $\beta(b)=\alpha(b)$. Then we say that a pair $(\alpha,\varphi_{[\beta]})$ is an $\F$-\emph{closed pair}, where $\varphi_{[\beta]}$ is the holonomy map in $\F$ associated to $\beta$.
In addition a pair $(\alpha,\varphi_{[\beta]})$ is called  $\F$-\emph{well closed pair} if in addition  $\alpha$ is regular in $\alpha(a)$ and $\alpha(b)$ and if $d\varphi_{[\beta]}df_{a}\alpha'(a)=d f_{b}\alpha'(b),$ where $f_{i}:\tub(P_{\alpha(i)})\rightarrow S_{\alpha(i)}$ is a submersion that describes the plaques in the neighborhood of $\alpha(i)$ for $i=a,b$. 
\end{definition}

Note that a horizontal periodic geodesic $\gamma$ is a well closed pair $(\gamma,\varphi_{[\beta]}).$

From now on, we assume that $\F$ is a Riemannian foliation with compact leaves on  a compact Riemannian manifold $M.$

We will see in the following that it is possible to assign a horizontal piecewise $\F$-periodic geodesic to a given $\F$-closed pair $(\alpha,\varphi_{[\beta]})$ (see Definition \ref{fclosed}). This process involves several  difficulties up to its definition. First we note that there exists a radius $\rho_0>0$ satisfying the following:
\begin{enumerate}
\item[ ({$ i$})] it is smaller than the injectivity radius of every point, 
\item[($ii$)] the balls $B(x,\rho_0)$ are always contained in a trivial neighborhood,
\item[($iii$)] there exists a unique minimizing horizontal geodesic between every point $x$ and every plaque for the trivial neighborhood of $(ii)$ at a distance lower than $\rho_0$.
\end{enumerate}

\subsubsection{$\hat{P}$-process}\label{P-process}

We are now ready to define the shortening process. Fix a real number $K>0$ and consider an $\F$-closed pair  $(\alpha,\varphi_{[\beta]})$ as in Definition \ref{fclosed} such that $E(\alpha)\leq K$. Given  a partition
\[a=l_{0}<l_{1}< \ldots < l_{k}=b\] such that $l_{i}-l_{i-1} < \frac{\rho^2_{0}}{K}$ for $i=1,\ldots, k$, Holder's inequality implies that
\begin{itemize}
\item $d(\alpha(l_{i-1}),\alpha(l_{i}))<\rho_0$, 
\item $\alpha|_{[l_{i-1},l_{i}]}$ is contained in a trivial neighborhood of $\F$,
\item there exists a unique minimizing horizontal geodesic $\tilde{\gamma}_i:[l_{i-1},l_{i}]\rightarrow M$ joining $\alpha(l_{i-1})$ and the plaque in the trivial neighborhood containing $\alpha(l_{i})$ and that satisfies $E(\tilde{\gamma}_i)\leq E(\alpha|_{[l_{i-1},l_{i}]})$.
\end{itemize}
Therefore, we can construct a piecewise ``disconnected'' horizontal geodesic from the curve $\alpha$. Now we will use the trivial holonomy in every trivial neighborhood to obtain a connected piecewise horizontal geodesic, so as a holonomy between the endpoints.

For $1\leq i \leq k$, let $\hat{\gamma}_{i,i}$ be the  minimizing segment of geodesic orthogonal to the plaque $P_{\alpha(l_{i})}$ such that $\hat{\gamma}_{i,i}(l_{i-1})=\alpha(l_{i-1})$ and $\hat{\gamma}_{i,i}(l_{i})\in P_{\alpha(l_{i})}$. Let $\hat{\beta}_{i,i}$ be a curve from $[l_{i-1},l_i]$ into $P_{\alpha(l_{i})}$ such that
$\hat{\beta}_{i,i}(l_{i-1})=\hat{\gamma}_{i,i}(l_{i})$ and $\hat{\beta}_{i,i}(l_i)=\alpha(l_{i})$.
Assume that $\hat{\beta}_{n-1,j}$ and $\hat{\gamma}_{n,j+1}$ are defined, then $\hat{\gamma}_{n,j}:=\|_{\hat{\beta}^{-1}_{n-1,j}}(\hat{\gamma}_{n,j+1})$  and $\hat{\beta}_{n,j}:=\eta(\hat{\gamma}_{n,j},\hat{\beta}_{n-1,j})$. Apply this process inductively for  $n=2,\dots,k$ and $j=n-1,\ldots,1$.
Finally define $\tilde{\beta}:=\hat{\beta}_{k,1}*\cdots*\hat{\beta}_{k,k}$, the piecewise horizontal geodesic $\hat{\gamma}=\hat{\gamma}_{1,1}*\hat{\gamma}_{2,1}\cdots*\hat{\gamma}_{k,1}$,  and the holonomy of  the endpoints by the curve $\hat{\beta}=\tilde{\beta}^{-1}\star\beta$. Summing up, given the $\F$-closed pair $(\alpha,\varphi_{[\beta]})$  and a family of nodes $a=l_{0}<l_{1}< \ldots < l_{k}=b$ such that $l_{i}-l_{i-1} < \frac{\rho^2_{0}}{K}$ for $i=1,\ldots, k$, we have obtained an $\F$-closed pair $\hat{P}(\alpha,\varphi_{[\beta]})=(\hat{\gamma},\varphi_{[\hat{\beta}]})$ such that $\hat{\gamma}$ is a piecewise horizontal geodesic with $E(\gamma)\leq K$ and $\hat{\beta}$ is a curve in $L_{\hat{\gamma}(a)}$ that joins the endpoints of $\hat{\gamma}$.

\subsubsection{The double shortening map}
\label{Subsec-double-shortening-map}
 As usual we will alternate two families of nodes in the shortening process to obtain a smooth curve in the limit. Choose two partitions $\{t_i\}$ and $\{\tau_i\}$ with $i=1,\ldots,k$ such that
\[\tau_0=\tau_k-1<t_{0}=0 < \tau_{1}< t_{1} <\tau_{2} <t_2< \ldots < \tau_{k} < t_{k}=1\] and $t_{i}-t_{i-1},  \tau_{i}-\tau_{i-1} < \frac{\rho^2_{0}}{K}$ for $i=1,\ldots, k.$ Given an $\F$-closed pair $(\alpha,\varphi_{[\beta]})$ as in the preceding subsection with $\alpha$ defined in $[0,1]$, we can apply the $\hat{P}$-process with the partition $0=t_0<\ldots<t_k=1$, obtaining a horizontal piecewise geodesic $\hat{\gamma}$ and a curve $\hat{\beta}$ in the leaf $L_{\hat{\gamma}(0)}$ joining the endpoints of $\hat{\gamma}$.  Now we can extend $\hat{\gamma}$ by parallel transport  to $[\tau_0,0]$ as follows: 
\begin{equation}
\label{Equation1-Subsec-double-shortening-map}
\hat{\gamma}(t):=||_{\hat{\beta}^{-1}}(\hat{\gamma}|_{[\tau_k,1]})(t+1).
\end{equation}
Moreover, we can bring the holonomy $\varphi_{[\hat{\beta}]}$ along $\hat{\gamma}|_{[\tau_0,0]}$ using the endpoint map $\eta$ (see Definition \ref{paralleltransport}) obtaining a holonomy $\varphi_{[\bar{\beta}]}$ in the leaf of $\hat{\gamma}(\tau_0)$  with $\bar{\beta}(0)=\hat{\gamma}(\tau_0)$ and $\bar{\beta}(1)=\hat{\gamma}(\tau_k)$. We can apply again the $\hat{P}$-process to the curve $\hat{\gamma}:[\tau_0,\tau_k]\rightarrow M$ and the holonomy $\varphi_{[\bar{\beta}]}$ obtaining $\hat{P}(\hat{\gamma},\varphi_{[\bar{\beta}]})=(\gamma_0,\varphi_{[\bar{\beta_0}]})$. 
Finally we extend the curve $\gamma_0$ to $[\tau_k,1]$ as
\begin{equation}
\label{Equation2-Subsec-double-shortening-map}
\gamma_0(t):=||_{{\bar\beta}_0}(\gamma_0|_{[\tau_0,0]})(t-1),
\end{equation}
and we consider in the leaf of $\gamma_0(0)$ the holonomy given by the endpoint map $\eta$ of $\bar{\beta_0}$ along $\gamma_0|_{[\tau_0,0]}$ obtaining an $\F$-closed pair $(\gamma_0,\varphi_{[\beta_0]})$.
 Therefore, we have obtained a double shortening map, that is, $P_0(\alpha,\varphi_{[\beta]})=(\gamma_0,\varphi_{[\beta_0]})$.


\subsubsection{Main propositions}

\begin{proposition}
\label{proposition-queda-energia-encurtamento}
 Let $(\alpha,\varphi_{[\beta]})$ be  an $\F$-closed pair  (with $\alpha:[0,1]\rightarrow M$) such that  $E(\alpha)\leq K$ and  
$P_{0}(\alpha,\varphi_{[\beta]})=(\gamma_0,\varphi_{[\beta_0]})$. Then $E(\gamma_{0}|_{[0,1]})\leq E(\alpha)$ with equality if and only if $\alpha$ is a horizontal periodic geodesic.
\end{proposition}
\begin{proof}
We have already observed in Subsection \ref{P-process} that a shortening $\hat{P}(\alpha,\varphi_{[\beta]})=(\hat{\gamma},\varphi_{[\hat{\beta}]})$  satisfies $E(\hat{\gamma})\leq E(\alpha)$. As the geodesic segments of $\hat{\gamma}$ are the unique minimizing geodesics joining the initial point with the plaque of the endpoint, the equality holds if and only if $\alpha$ is a piecewise geodesic with nodes $t_0,\ldots,t_{n-1}$, and in this case $\hat{\gamma}=\alpha$. In the $P_0$-process we apply twice the $\hat{P}$-process. As we change the nodes  and $E(\hat{\gamma}|_{[\tau_0,\tau_k]})=E(\hat{\gamma}|_{[0,1]})$, the energy of $\gamma_0$ remains the same if and only if $\alpha$ is a geodesic such that the extension to $[\tau_0,0]$ by the parallel transport along $\beta$ gives a geodesic $\gamma_0$ in $[\tau_0,1]$.
\end{proof}

In the following, we will say that a curve $\alpha:[a,b]\rightarrow M$ is {\it $\F$-closed} if the endpoints are in the same leaf of $\F$.  We say that two $\F$-closed curves are {\it $\F$-homotopic} if there exists a homotopy between them by $\F$-closed curves.

 The fact that the restriction of the considered curves to the partitions $\{t_{i}\}$ and $\{\tau_{i}\}$ are contained in trivial neighborhoods of $\F$ and the equifocality of $\F$ imply the next lemma. 

\begin{lemma}
\label{homotopia-base}
Let $(\alpha,\varphi_{[\beta]})$ be an $\F$-closed pair such that $E(\alpha)\leq K$ and  
$\hat{P}(\alpha,\varphi_{[\beta]})=(\hat{\gamma},\varphi_{[\hat{\beta}]})$. Then $\hat{\gamma}$ is $\F$-homotopic to $\alpha$.
\end{lemma}

\begin{proposition}\label{proposition-homotopia-encurtamento-curvaoriginal}
Let $(\alpha,\varphi_{[\beta]})$ be an $\F$-closed pair (with $\alpha:[0,1]\rightarrow M$) such that $E(\alpha)\leq K$ and  
$P_{0}(\alpha,\varphi_{[\beta]})=(\gamma_0,\varphi_{[\beta_0]})$. Then $\gamma_0$ is $\F$-homotopic to $\alpha$.
\end{proposition}

\begin{proof}
By applying Lemma \ref{homotopia-base} we obtain that $\alpha$ is $\F$-homotopic to the first shortening $\hat{\gamma}$. As we extend $\hat{\gamma}$ with the holonomy $\varphi_{[\hat{\beta}]}$, we have that $\hat{\gamma}(t)$ and $\hat{\gamma}(t+1)$ are in the same   leaf for $t\in[\tau_0,0].$
It also follows from Lemma \ref{homotopia-base} that there exists a map $H$ (that we call $\F$ homotopy) defined in $[\tau_{0},\tau_{k}]\times[0,1]$ such that
\begin{enumerate}
\item $H(\tau_{k},s)\in L_{H(\tau_{0},s)}$ for each $s\in[0,1],$
\item $H(\cdot,0)=\hat{\gamma}|_{[\tau_{0},\tau_{k}]}$ and $H(\cdot,1)=\gamma_{0}|_{[\tau_{0},\tau_{k}]}$.
\end{enumerate} 
By transporting  horizontal  segments of geodesics, the $\F$-homotopy $H$ can be  chosen to admit an extension to $[\tau_{0},1]\times[0,1]$ and so that $H(1,s)\in L_{H(0,s)}$ for each $s$.  Therefore $\gamma_{0}|_{[0,1]}$ is $\F$-homotopic to $\hat{\gamma}|_{[0,1]}$ and hence $\F$-homotopic to $\alpha$.
\end{proof}

We will denote by $\Pi_1$ and $\Pi_2$ respectively the first and the second projections of an $\F$-closed pair. Given a closed curve, if nothing is said, we will assume that it is an $\F$-closed pair considering the trivial holonomy.

\begin{proposition}
\label{proposition-interecao-encurtamento}
Let $\alpha:[0,1]\rightarrow M$ be a closed curve with $E(\alpha)\leq K$. Then a subsequence of $\Pi_1\circ P_0^{n}\alpha$ converges uniformly to a (possibly trivial) horizontal periodic geodesic.
\end{proposition}
\begin{proof}
Each curve $\Pi_1\circ P_0^{n}\alpha$ is a horizontal  periodic piecewise geodesic with nodes $\Pi_1\circ P_0^{n}\alpha(\tau_{1}),\ldots, \Pi_{1}\circ P_0^{n}\alpha(\tau_{k})$. Note that each such curve may be identified with a $k-$tuple $(\Pi_1\circ P_0^{n}\alpha(\tau_{1}),\ldots,\Pi_1\circ P_0^{n}\alpha(\tau_{k}))\in M^{k}:=M\times\ldots\times M.$ Since $M^{k}$ is compact, a subsequence of these nodes converges to some $(p_{1},\ldots,p_{k})\in M^{k}$ and by the continuity of the exponential map, a subsequence $\gamma_m$ of $\Pi_1\circ P_0^{n}\alpha$ converges uniformly towards the horizontal piecewise geodesic $\gamma_{0}$ with nodes $\gamma_{0}(\tau_{i})=p_{i}$ and such that
\[(\gamma_{m+1},\varphi_{[\beta_{m+1}]})=P_0^{\mu(m)}(\gamma_m,\varphi_{[\beta_m]})\]
with $\mu(m)\geq 1$. 

We will see that the holonomies $\varphi_{[\beta_m]}$ admit a ``constant'' subsequence in a certain sense. 

According to  Molino \cite[Lemma 3.7]{Molino} we can choose a   radius $\epsilon<1$ so that: 
\begin{itemize}
\item The tubular neighborhood $\tub_{\epsilon}(L_{\gamma_{0}(0)})$ is saturated by leaves;
\item for all $x\in L_{\gamma_{0}(0)}$ the slice $S_{x}$ (of radius $\epsilon$), defined as 
\[S_{x}:=\{\exp_{x}(\xi)| \xi\in \nu P_{x}, \|\xi\|<\epsilon\},\] is transversal to the foliation; 
\item if $\tilde{L}$ is a leaf in $\tub_{\epsilon}(L_{\gamma_{0}(0)})$ then the points of $\tilde{L}$ are all at the same distance from $L;$
\item for each $x\in L_{\gamma_{0}(0)}$ there exists a plaque $P_{x}$ such that $\pi^{-1}(P_{x})$ is a simple open set, where $\pi:\tub_{\epsilon}(L_{\gamma_{0}(0)})\rightarrow L_{\gamma_{0}(0)}$ is the
radial projection. 
\end{itemize}

Choose $N_{0}$ such that if $m>N_{0}$ then
$L_{\gamma_{m}(0)}\subset \tub_{\frac{\epsilon}{2}}(L_{\gamma_{0}(0)}).$ Let $\beta_{m}$ be a representative for the holonomy class $\Pi_{2}\circ P_0^{m}\alpha \in L_{\gamma_{m}(0)}$  and define $\tilde{\beta}_{m}:=\pi(\beta_{m})\in L_{\gamma_{0}(0)}$, where $\pi:\tub_{\frac{\epsilon}{2}}(L_{\gamma_{0}(0)})\rightarrow L_{\gamma_{0}(0)}$ is the radial projection.


Our choice of $\epsilon$, the fact that the holonomy group of each leaf is finite and properties of the holonomy maps imply the next lemma.

\begin{lemma}
\label{lemma-gamma0-well-closed}
In the above situation:
\begin{enumerate}
\item[(a)]there exists a holonomy $\varphi_{[\beta_{0}]}$ in $L_{\gamma_{0}(0)}$ such that $(\gamma_{0},\varphi_{[\beta_{0}]})$ is a well closed pair,
\item[(b)] 
 there exists a subsequence $\tilde{\beta}_{m_{i}}$ such that  $\varphi_{[\beta_{0}]}=\varphi_{[\tilde{\beta}_{m_{i}}]}:S_{\gamma_{0}(0)}\rightarrow S_{\gamma_{0}(1)}.$
\end{enumerate}
\end{lemma}

For the sake of simplicity we will still denote the subsequence $m_i$ by $m$. 
It is easy to see that 
\begin{equation}\label{formula-energia}
E(\gamma_m)=\sum_{i=1}^{k}\frac{\dist(\gamma_m(\tau_{i-1}),\gamma_m(\tau_i))^2}{2(\tau_i-\tau_{i-1})},
\end{equation}
and then ${\lim}_{m\rightarrow\infty}E(\gamma_m)=E(\gamma_0).$ Therefore

\begin{align*}
E(\gamma_0)=&\lim_{m\rightarrow\infty}E(\gamma_{m+1})
=\lim_{m\rightarrow\infty}E\big(\Pi_1\circ P_0^{\mu(m)}(\gamma_{m},\varphi_{[\beta_m]})\big)\\
\leq &\lim_{m\rightarrow\infty}E\big(\Pi_1\circ P_0(\gamma_{m},\varphi_{[\beta_m]})\big)
\leq  \lim_{m\rightarrow\infty}E(\gamma_m)
= E(\gamma_0),
\end{align*}
where we have used Proposition \ref{proposition-queda-energia-encurtamento}.  
We conclude from the above equality that
\begin{equation}
\label{Eq-Fundamental-proposition-interecao-encurtamento}
\lim_{m\rightarrow\infty}E\big(\Pi_1\circ P_0(\gamma_m,\varphi_{[\beta_m]})\big)=E(\gamma_0).
\end{equation}
The fact that minimal segments of geodesics depend smoothly on  their endpoints, $\varphi_{[\beta_{0}]}=\varphi_{[\tilde{\beta}_{m}]}$ and that the energy is not changed by parallel transport of horizontal segments imply the next lemma.

\begin{lemma}
\label{lemma-continuidadeP0}
 $E\big(\Pi_1\circ P_0(\gamma_m,\varphi_{[\beta_m]})\big)$  converges to $E\big(\Pi_1\circ P_0(\gamma_0,\varphi_{[\beta_{0}]})\big)$.\qed
\end{lemma}


Lemma \ref{lemma-continuidadeP0} and \eqref{Eq-Fundamental-proposition-interecao-encurtamento}   imply 
\[E(\Pi_1\circ P_{0}(\gamma_0,\varphi_{[\beta_0]}))=\lim_{m\rightarrow\infty}E(\Pi_1\circ P_0(\gamma_m,\varphi_{[\beta_m]}))=E(\gamma_0),
\]
and from Proposition \ref{proposition-queda-energia-encurtamento} we  conclude that $\gamma_0$ is a horizontal periodic geodesic. 


\end{proof}

\begin{remark}
Shortening process and the above results can be straightforwardly generalized to Riemannian foliations with closed embedded leaves on a complete Riemannian  manifold. In this case Proposition \ref{proposition-interecao-encurtamento} should be reformulate as follows:
\emph{
Let $\alpha:[0,1]\rightarrow M$ be a closed curve with $E(\alpha)\leq K$. 
Then there exist a subsequence $\{\gamma_{m}\}$ of $\Pi_1\circ P_0^{n}\alpha$ and  a sequence $\{ k_{m}\}$ of holonomies of $L_{\gamma_{m}(0)}$ such that
$\{k_{m}\gamma_{m}\}$ converges   uniformly to a (possibly trivial) horizontal periodic geodesic.
}
\end{remark}


\subsection{Proof of the Theorem \ref{teo-RF-PeriodicGeodesic} }

\

(a)  let $\alpha$ be a loop of $M$ that is not free homotopic to any  loop contained in any leaf of $\F$.   
According to Proposition \ref{proposition-interecao-encurtamento} there exists a subsequence $\{\gamma_{n}\}$ of $\Pi_{1}P_{0}^{m}(\alpha)$ that converges to a (possibly trivial) horizontal geodesic $\gamma^{0}$.  Assume by contradiction that $\gamma^{0}=y$, i.e., that $\gamma^{0}$ is trivial.  

Consider $n$ big enough such that $\beta_{n}$ and $\gamma_{n}$ are in the same tubular neighborhood $\tub(L_{y})$. By the radial projection in the axis $L_{y}$ we can construct a curve $\hat{\delta}\subset L_{y}$ such that $\hat{\delta}$ is free homotopic to $\beta_{n}\star\gamma_{n}$. Since $\beta_{n}\star\gamma_{n}$ is free homotopic to $\alpha$  (see Proposition \ref{proposition-homotopia-encurtamento-curvaoriginal}), we conclude that
$\alpha$ is free homotopic to $\hat{\delta}.$ This contradicts the hypothesis of item (a). Therefore   
$\gamma^{0}$ is a  nontrivial horizontal periodic geodesic.


(b) Consider a sequence of loops $[\alpha_{n}]\in\pi_{1}(M,x_{0})$ such that $\alpha_i$ is not free homotopic to $\alpha_j$. 
It follows from Proposition \ref{proposition-interecao-encurtamento} that, for each $i$ there exists a subsequence $\{ \gamma_{n}^{i}\}$ of $\Pi_{1}P_{0}^{m}(\alpha_{i})$ that converges to a (possibly trivial) horizontal geodesic $\gamma^{i}$.  Assume by contradiction that  $\gamma^{i}$ is a point $y_{i}$ for each $i\in\mathbb{N}$ .

For a fixed $i$ consider $n$ big enough such that $\beta_{n}$ and $\gamma_{n}^{i}$ are in the same tubular neighborhood $\tub(L_{y_{i}})$.  
By the radial projection in the axis $L_{y_i}$ we can construct a curve $\hat{\delta}_{i}\subset L_{y_{i}}$ such that $\hat{\delta}_{i}$ is free homotopic to $\beta_{n}\star\gamma_{n}^{i}$. Since $\beta_{n}\star\gamma_{n}^{i}$ is free homotopic to $\alpha_{i}$ we conclude that
$\alpha_{i}$ is free homotopic to $\hat{\delta}_{i}.$ 
Since the sequence $\{y_{i}\}$ is contained in the compact space $M$, there exists $y\in M$ and a subsequence (that we also denote by $\{y_{i}\}$) such that $y_{i}\rightarrow y.$ Therefore, for  $i$ big enough, we have that
$\hat{\delta}_{i}\subset L_{y_{i}}\subset \tub(L_{y})$. By the radial projection in the axis $L_{y}$ we can construct a curve $\tilde{\delta}_{i}\subset L_{y}$ such that $\tilde{\delta}_{i}$ is free homotopic to $\hat{\delta}_{i}$ and $\tilde{\delta}_{i}(0)=y$.
Since $\alpha_{i}$ is free homotopic to $\hat{\delta}_{i}$ and $\hat{\delta}_{i}$ is free homotopic to $\tilde{\delta}_{i}$ we conclude that
$\alpha_{i}$ is free homotopic to $\tilde{\delta}_{i}$.
On the other hand the cardinality of $\pi_{1}(L,y)$ is finite, and  we conclude (for a subsequence) that $\tilde{\delta}_{i}$ is homotopic to $\tilde{\delta}_{j}$. Hence $\alpha_{i}$ is free homotopic to $\alpha_{j}$ and this contradicts the hypothesis of item (b).
Therefore, there exists  some $i_{0}$ so that   $\gamma^{i_{0}}$ is a nontrivial horizontal periodic geodesic.



\section{Appendix}
In this appendix we present the shortening process in the special case of  Riemannian good orbifolds and give another proof of  
Theorem \ref{theorem-ClosedGeodesicGoodOrbifold}.


\subsection{Shortening process}

\label{shortening-orbifold} 

In this subsection we construct the shortening process in the good orbifold.  From now on, we assume that $\Sigma$  is a connected complete Riemannian manifold and $W$ is an infinite discrete subgroup of isometries of $\Sigma$ such that $\Sigma/W$ is compact.  We also assume that  the action of $W $ on $\Sigma$ is proper.

  Since $\Sigma/W$ is compact, there exists  a radius $\rho_0>0$ such that for each $x\in \Sigma$ and $q\in B_{\rho_{0}}(x)$, the shortest segment of geodesic from $x$ to $q$ is unique. 

We still use the conventions of Remark \ref{remark-convention-concatenation}.

\subsubsection{$\hat{P}$-process}
\label{P-process-orbifold}

 Fix a real number $K>0$ and consider a curve $\alpha:[a,b]\rightarrow \Sigma$ and an element $w^{0}\in W$ so that
 $\alpha(b)=w^{0}\alpha(a)$ and   $E(\alpha)\leq K$. Given  a partition
\[a=l_{0}<l_{1}< \ldots < l_{k}=b\] such that $l_{i}-l_{i-1} < \frac{\rho^2_{0}}{K}$ for $i=1,\ldots, k$, Holder's inequality implies that
\begin{itemize}
\item $d(\alpha(l_{i-1}),\alpha(l_{i}))<\rho_0$, 
\item there exists a unique minimizing  geodesic $\tilde{\gamma}_i:[l_{i-1},l_i]\rightarrow\Sigma$ joining $\alpha(l_{i-1})$ and  $\alpha(l_{i})$ and that satisfies $E(\tilde{\gamma}_i)\leq E(\alpha|_{[l_{i-1},l_i]})$.
\end{itemize}
Therefore, we can construct a piecewise  geodesic from the curve $\alpha$ as $\hat{P}(\alpha)=\gamma_{1}*\cdots*\gamma_{k}$. Note that $\hat{P}(\alpha)(b)=w^{0}\hat{P}(\alpha)(a)$.

\subsubsection{The double shortening map}
\label{Subsec-double-shortening-map-orbifold}
 As usual we will alternate two families of nodes in the shortening process to obtain a smooth curve in the limit. Choose two partitions $\{t_i\}$ and $\{\tau_i\}$ with $i=1,\ldots,k$ such that
\[\tau_0=\tau_k-1<t_{0}=0 < \tau_{1}< t_{1} <\tau_{2} <t_2< \ldots < \tau_{k} < t_{k}=1\] and $t_{i}-t_{i-1},  \tau_{i}-\tau_{i-1} < \frac{\rho^2_{0}}{K}$ for $i=1,\ldots, k.$ Given a  pair $(\alpha,w^{0})$ as in the preceding subsubsection with $\alpha$ defined in $[0,1]$, we can apply the $\hat{P}$-process with the partition $0=t_0<\ldots<t_k=1$, obtaining a  piecewise geodesic $\hat{\gamma}$ with $\hat{\gamma}(1)=w^{0}\hat{\gamma}(0)$.

  Now we can extend $\hat{\gamma}$   to $[\tau_0,0]$ as follows: 
\begin{equation}
\label{Equation1-Subsec-double-shortening-map-orbifold}
\hat{\gamma}(t):=(w^{0})^{-1}(\hat{\gamma}|_{[\tau_k,1]})(t+1).
\end{equation}
 We can apply again the $\hat{P}$-process to the curve $\hat{\gamma}:[\tau_0,\tau_k]\rightarrow \Sigma$  obtaining a curve $\gamma_{0}:[\tau_0,\tau_k]\rightarrow \Sigma$ such that $\gamma_{0}(\tau_{1})=w^{0}\gamma_{0}(\tau_{0})$. 
Finally we extend the curve $\gamma_0$ to $[\tau_k,1]$ as
\begin{equation}
\label{Equation2-Subsec-double-shortening-map-orbifold}
\gamma_0(t):={w}^{0}(\gamma_0|_{[\tau_0,0]})(t-1).
\end{equation}
Therefore, we have obtained a double shortening map, that is, $P_0(\alpha)=\gamma_0$, 
where $\gamma_{0}:[\tau_{0},1]\rightarrow \Sigma$ and  $\gamma_{0}(t+1)=w^{0}\gamma_{0}(t)$ for $t\in[\tau_{0},0]$.

We can simplify the arguments of the last section and prove the next two propositions. 

\begin{proposition}
\label{proposition-queda-energia-encurtamento-orbifold}
 Let $\alpha:[0,1]\rightarrow \Sigma$ be a curve in $\Sigma$ such that $\alpha(1)=w^{0}\alpha(0)$ and  $E(\alpha)\leq K$. Set
 $\gamma_{0}:=P_{0}(\alpha)$. Then $E(\gamma_{0}|_{[0,1]})\leq E(\alpha)$ with equality if and only if $\alpha$ is a closed  geodesic of $\Sigma/W$.
 \end{proposition}

\begin{proposition}
\label{proposition-interecao-encurtamento-orbifold}
Let $\alpha:[0,1]\rightarrow \Sigma$ be a  curve with $\alpha(1)=w^{0}\alpha(0)$ and $E(\alpha)\leq K$. 
Then there exist a subsequence $\{\gamma_{m}\}$ of  $\gamma_{n}= P_0^{n}\alpha$ and a sequence $\{ k_{m}\}$ in $W$ such that \{$k_{m}\gamma_{m}\}$ converges to a closed geodesic $(\gamma_{0},w)$. In other words, a subsequence of  classes of $\{ P_0^{n}\alpha\}$ converges uniformly to a (possibly trivial) closed geodesic of $\Sigma/W$.
\end{proposition}

\subsection{Proof of Theorem \ref{theorem-ClosedGeodesicGoodOrbifold}}

 Using shortening process and  standard properties of proper actions 
 we provide another proof of Theorem \ref{theorem-ClosedGeodesicGoodOrbifold} that we reformulate as follows.

\begin{theorem}
\label{theorem-ClosedGeodesicGoodOrbifold2}
Let $\Sigma$ be a connected complete Riemannian manifold and $W$ be an infinite discrete subgroup of isometries of $\Sigma$ whose action on $\Sigma$ is proper and such  that the good orbifold $\Sigma/W$ is compact.  
Assume  that there exists an element $w^{0}\in W$ that does not fix points (e.g. $w^{0}$ has infinite order).
Then for each curve $\alpha$ joining a point $p$ to $w^{0}p$, a subsequence of classes of iterations of double shortening of $\alpha$  converges to a nontrivial closed geodesic in $\Sigma/W.$ 
\end{theorem}

\begin{proof} 

Let $\alpha:[0,1]\rightarrow \Sigma$ be a curve so that $\alpha(0)=p$ and $\alpha(1)=w^{0}p$.
Set $x_{m}:=P_{0}^{m}\alpha(0)$ and recall that $P_{0}^{m}\alpha(1)=w^{0}x_{m}.$
Proposition \ref{proposition-interecao-encurtamento-orbifold} implies that there exists a sequence  $\{k_{n}\}\subset W$  and a subsequence  $\{\gamma_{n}\}$ of $P_{0}^{m}\alpha$ such that 
$k_{n}\gamma_{n}$ converges uniformly to a (possibly trivial) closed geodesic $\gamma$ of $\Sigma/W.$ Set $x:=\gamma(0)$ and $y:=\gamma(1)$ and note that 
\begin{equation}
\label{eq-2-theorem-ClosedGeodesicGoodOrbifold}
k_{n}x_{n}\rightarrow x ; \ \, k_{n}w^{0}x_{n}\rightarrow y .
\end{equation}

Suppose  that the closed geodesic $\gamma$ is trivial, i.e., $x=y$.
Now set $g_{n}:=k_{n}(w^{0})^{-1}(k_{n})^{-1}\in W$.  Eq. (\ref{eq-2-theorem-ClosedGeodesicGoodOrbifold}) implies  
 \begin{equation}
 \label{eq-3-theorem-ClosedGeodesicGoodOrbifold} 
   g_{n}(k_{n}w^{0}x_{n})\rightarrow x ; \ \, (k_{n}w^{0}x_{n})\rightarrow x .
\end{equation}
Since the action of $W$ on $\Sigma$ is proper, Eq. (\ref{eq-3-theorem-ClosedGeodesicGoodOrbifold}) implies that there exists a subsequence $\{g_{n}\}$, 
such that $g_{n}\rightarrow g\in W_{x}$. Since $W$ is discrete,  there exists  $n\in\mathbb{N}$ such that $g_{n}=g\in W_{x}$ and hence
\begin{equation}
 \label{eq-4-theorem-ClosedGeodesicGoodOrbifold}
 w^{0}=(k_{n})^{-1}(g)^{-1}(k_{n})\in W_{(k_{n})^{-1}x  } .
\end{equation}
The above equation contradicts the assumption that $w^{0}$ does not fix  points. Therefore $\gamma$ is  a nontrivial closed geodesic.

\end{proof}





\bibliographystyle{amsplain}

\end{document}